\documentclass{amsart}
\usepackage{amsthm}
\usepackage{amssymb}
\usepackage{amsrefs}
\usepackage{enumerate}
\usepackage[matrix,arrow]{xy}
\usepackage{comment}

\numberwithin{equation}{section}
\theoremstyle{plain}
\newtheorem{thm}[equation]{Theorem}
\newtheorem{cor}[equation]{Corollary}

\newtheorem*{thm*}{Theorem}

\theoremstyle{definition}

\newtheorem{rem}[equation]{Remark}

\begin{document}

\title{A generalization of the 3d distance theorem}

\author{Manish Mishra \and Amy Binny Philip}

\address{Indian Institute of Science Education and Research Pune, Dr. Homi
Bhabha Road, Pasha, Pune 411008, Maharashtra, India}

\email{manish@iiserpune.ac.in\\
 amy@iiserpune.ac.in}
\begin{abstract}
Let $P$ be a positive rational number. Call a function $f:\mathbb{R}\rightarrow\mathbb{R}$ to have \textit{finite gaps property mod $P$} if the following holds: for any positive irrational $\alpha$ and positive integer $M$, when the values of $f(m\alpha)$, $1\leq m\leq M$, are inserted mod $P$ into the interval $[0,P)$ and arranged in increasing order, the number of distinct gaps between successive terms is bounded by a constant
$k_{f}$ which depends only on $f$. In this note, we prove a generalization of the 3d distance theorem of Chung and Graham. As a consequence, we show that a piecewise linear map with rational slopes and having only finitely many non-differentiable points has finite gaps property mod $P$. We also show that if $f$ is distance to the nearest integer function, then it has finite gaps property mod $1$ with $k_f\leq6$.

\end{abstract}

\maketitle

\section{introduction}

The well known three gaps theorem was first observed by H. Steinhaus
and proved independently by V. T. S{\'o}s \cite{sos1957theory, sos1958distribution} and  others \cite{swierczkowski1958successive, sur1958nyi} 
(see \cite{BAASJ17} for a nice summary and recent generalization). The three gaps theorem is a special case ($d=1$) of the following 
more general theorem of Chung and Graham \cite{grahamchung1976}.
\begin{thm*}
[3d distance theorem] Let $\alpha>0$ be an irrational number and
$N_{1},\ldots,N_{d}$ be positive integers. When the fractional parts
of $d$ arithmetic sequences $n\alpha+k_{i}$, $k_{i}\in\mathbb{R},$
$n=1,\ldots,N_{i}$, $1\leq i\leq d$ are inserted into a circle of
unit circumference, the gaps between successive terms takes at most
$3d$ distinct values. 
\end{thm*}

What makes this theorem surprising is the fact that the fractional parts of the sequence $n\alpha$ are known to be uniformly 
distributed in the interval $[0,1)$.

For a positive rational $\lambda$, let $I_{\lambda}$ be the discrete
set $\{n\lambda\mid n\in\mathbb{Z}\}$. For $x\in\mathbb{R}$, define $\lambda$-floor $\lfloor{x}{\rfloor}_{\lambda}$ and $\lambda$-roof $\lceil{x}{\rceil}_{\lambda}$ as:
\[
\lfloor{x}{\rfloor}_{\lambda}=\mathrm{max}\{r\in I_{\lambda}\mid r\leq x\}, \\
\lceil{x}{\rceil}_{\lambda}=\mathrm{min}\{r\in I_{\lambda}\mid r\geq x\}.
\]
 Define $\lambda$-fractional part functions $\{-\}'_\lambda$ and $\{-\}''_\lambda$ as: 
\[
\{x\}_{\lambda}'=
\begin{cases}
x-\lfloor{x}{\rfloor}_{\lambda} & \text{if } x\geq 0\\
x -\lceil{x}{\rceil}_{\lambda} & \text{if } x<0,

\end{cases}
\] 
\[
\{x\}_{\lambda}''=x-\lfloor{x}{\rfloor}_{\lambda}.
\]
Define $\{x\}'_\infty=x$. We write $\{x\}''_{1}$ as $\{x\}$.  

Choose a $\lambda$-fractional part function $\{-\}'_\lambda$ or $\{-\}''_\lambda$ and denote it by $\{-\}_{\lambda}$. We prove the following generalization of the 3d distance theorem.
\begin{thm}
\label{thm1}Let $\alpha>0$ be an irrational number, $N_{1},\ldots,N_{d}$
be positive integers and $n_{1},\ldots,n_{d}$ be non-negative integers such that
$n_{i}\leq N_{i}$, $1\leq i\leq d$. Write $N=\sum_{i=1}^{d}(N_{i}-n_{i})$. Consider the linear
maps $\hat{f}_{i}:x\in\mathbb{R}\mapsto\frac{p_{i}}{q}x+k_{i}\in\mathbb{R}$,
$1\leq i\leq d$, where $0\neq p_{i}\in\mathbb{Z}$, $q\in\mathbb{Z}_{>0}$
and $k_{i}\in\mathbb{R}$. Fix $P$ to be a positive rational and let $\lambda$ be any positive integer multiple of $Pq$. Define $f_{i}:\mathbb{R}\rightarrow\mathbb{R}$ by $f_{i}(x)=\hat{f}_{i}(\{x\}_{\lambda})$. Insert mod $P$, the values of $f_{i}(m\alpha)$, $n_{i} < m\leq N_{i}$,
$1\leq i\leq d$, in the interval $[0,P)$ to form an increasing sequence
$(b_{n})_{1\leq n\leq N}$. Write $\ell=\mathrm{lcm}(p_{1},\ldots,p_{d})>0$,
$c_{i}:=\ell/p_{i}$ and $c=\sum_{i=1}^{d}|c_{i}|$. Then there are
at most $3c$ distinct values in the set of gaps $g_{m}$ defined
by 
\[
g_{1}=P+b_{1}-b_{N},\qquad g_{m}=b_{m}-b_{m-1},\qquad m=2,\ldots,N.
\]
\end{thm}
Note that Theorem \ref{thm1} allows the possibility of some points
to coincide. The ordering of coincidental points is defined in Section
\ref{sec1}.

Let $||\cdot||:\mathbb{R}\rightarrow[0,1/2]$ denote the distance
to the nearest integer function. By definition
\[
||x||=\mathrm{min}(\{|x|\},1-\{|x|\}).
\]
As a special case
of Theorem \ref{thm1}, we obtain the following result which was proved in 
\cite{henk} using different methods. 

\begin{cor}
\label{corr2}Let $\alpha>0$ be an irrational number and $M>1$ be
an integer. When the values $||n\alpha||$, $1\leq n\leq M$, are
arranged in ascending order in the interval $[0,\frac{1}{2}]$, the
gaps between successive terms may take at most $6$ distinct values. 
\end{cor}

Our proof of Corollary \ref{corr2} is significantly shorter than the proof in 
\cite{henk}. However, the bound obtained in loc. cit. is effective. 

\begin{cor}
\label{thm:2}Let $f:\mathbb{R}\rightarrow\mathbb{R}$ be a 
piecewise linear map with rational slopes and having only finitely many non-differentiable points. Let $\alpha>0$
be an irrational number and $M>1$ be an integer. For any positive rational $P$, when the values
$f(m\alpha)$, $1\leq m\leq M$, are inserted $\text{mod } P$ in $[0,P)$ and arranged in ascending order,
the gaps between successive terms may take at most $k_{f}$ distinct
values, where $k_{f}$ is a constant which depends only on $f$. 
\end{cor}
Our proof of Theorem \ref{thm1} is an adaptation of the elegant proof 
of the $3d$ distance Theorem by Liang \cite{liang1979}.

\section{\label{sec1}Proof of Theorem \ref{thm1}}

\begin{proof}For $1\leq i\leq d$, let $\mathcal{B}_{i}$ be the set of all triples
$\beta_{im}=(\gamma_{im},i,m)\in[0,P)\times\mathbb{Z}\times\mathbb{Z}$
where $\gamma_{im}\cong f_{i}(m\alpha)\text{ mod }P$,
$n_i\leq m\leq N_{i}$. Write $\mathbb{B}=\bigcup_{i=1}^{d}\mathcal{B}_{i}$.
We give a strict ordering $\prec$ on $\mathbb{B}$ by declaring $\beta_{im}\prec\beta_{jn},$
iff 
\[
\text{\ensuremath{\gamma}}_{im}<\gamma_{jn}\qquad\text{or}\qquad\gamma_{im}=\gamma_{jn}\text{ and }[i<j\text{ or }(i=j\text{ and }m<n)].
\]
 Arrange the elements of $\mathbb{B}$ in a strictly increasing sequence
$(b_{n})_{1\leq n\leq N}$ with this ordering. Applying arithmetic
modulo $P$, we identify $P$ with $0$ and consider $\{\gamma_{im}\mid n_i\leq m\leq N_i, 1\leq i \leq d\} $
as living in this circle $[0,P]$. This makes the ordering on $\mathbb{B}$
a cyclic ordering, which we again denote by $\prec$. Thus $b_{N}$
and $b_{1}$ are consecutive in this cyclic ordering. To simplify notation, will often abuse notation and write $\beta_{im}$ when we mean $\gamma_{im}$.

A \textit{gap interval} is an interval in the circle $[0,P]$ of the form $[\beta_{in},\beta_{jm}]$
where $\beta_{in}$, $\beta_{jm}$ are consecutive points of $\mathbb{B}$
in the cyclic ordering $\prec$. Write $\ell_{0}=\ell/q$. A gap interval
is \textit{rigid} if translating a gap interval by $\ell_{0}\alpha$
does not produce a gap interval. Observe that gap intervals cannot
loop upon successive translations by $\ell_{0}\alpha$. To see this,
suppose $s$ is a positive integer such that translation by $s\ell_{0}\alpha$
maps $[\beta_{in},\beta_{jm}]$ to itself. Then either $\beta_{in}+s\ell_{0}\alpha=\beta_{in}$
and $\beta_{jm}+s\ell_{0}\alpha=\beta_{jm}$, or $\beta_{in}+s\ell_{0}\alpha=\beta_{jm}$
and $\beta_{jm}+s\ell_{0}\alpha=\beta_{in}$. Either of these cases
contradicts the irrationality of $\alpha$. 

Now, a gap interval $I$ is rigid if upon translation by $\ell_{0}\alpha$
it produces an interval $J$ for which one of the following holds:

\begin{enumerate}

\item[(i)]At least one of the end points of $J$ is not in $\mathbb{B}$.

\item[(ii)] The translated interval $J$ has endpoints in $\mathbb{B}$
but they are not consecutive.

\end{enumerate}

For case (i), let $\beta_{in}$ be an end point of $I$ such that
$\beta_{in}+\ell_{0}\alpha=\beta_{i(n+c_{i})}\notin\mathbb{B}$. Then
in particular, $\beta_{i(n+c_{i})}\notin\mathcal{B}_{i}$. If $c_{i}>0$,
then $\beta_{i(n+c_{i})}\not\notin\mathcal{B}_{i}$ iff $n+c_{i}>N_{i}$.
Then, $\beta_{in}$ will be in the set
\[
S_{i}=\{\beta_{im}\mid N_{i}-c_{i}+1\leq m\leq N_{i}\}.
\]

If $c_{i}<0$, then $\beta_{i(n+c_{i})}\notin\mathcal{B}_{i}$ iff
$n+c_{i}\leq n_i$. Then $\beta_{in}$ will be in the set 
\[
T_{i}=\{\beta_{im}\mid1+n_i\leq m\leq-c_{i}+n_i\}.
\]

Call the elements of $S_{i}$ and $T_{i}$ to be \textit{starting
points}. Then for each $i$, the starting points have cardinality
$|c_{i}|$. Since each starting point is the boundary of at most two
gap intervals, case (i) contributes at most $2\sum_{i=1}^{d}|c_{i}|=2c$
rigid intervals. 

For case (ii), let $\beta_{kp}\in\mathbb{B}$ be an internal point
of $J$. Then $\beta_{k(p-c_{k})}$ is an internal point of $I$.
Since $I$ is a gap interval, this implies that $\beta_{k(p-c_{k})}\notin\mathbb{B}$.
In particular $\beta_{k(p-c_{k})}\notin\mathcal{B}_{k}$. This implies
that $\beta_{kp}$ belongs to the set 
\[
T_{k}^{\prime}=\{\beta_{km}\mid1+n_k\leq m\leq c_{k}+n_k\}
\]
or 
\[
S_{k}^{\prime}=\{\beta_{km}\mid N_{k}+c_{k}+1\leq m\leq N_{k}\}
\]
 according as $c_{k}>0$ or $c_{k}<0$. Call the elements of $S_{i}^{\prime}$
and $T_{i}^{\prime}$ to be \textit{finish points}. Then for each
$i$, the finish points have cardinality at most $|c_{i}|$. Thus
case (ii) contributes at most $\sum_{i=1}^{d}|c_{k}|=c$ rigid intervals. 

We have shown that there can be at most $3c$ distinct rigid intervals
and consequently at most $3c$ gap interval sizes. This completes
the proof. \end{proof}

\section{Proof of Corollaries \ref{corr2} and \ref{thm:2}}
\begin{proof}[Proof of Corollary \ref{corr2}]We retain the notations
of Theorem \ref{thm1} and its proof in Section \ref{sec1}. Put $d=2$, $q=1$
$p_{1}=1$, $p_{2}=-1$, $k_{1}=0$ , $k_{2}=1$, $N_{1}=N_{2}=M$, $P=1=\lambda$ and $\{-\}_\lambda=\{-\}''_\lambda$.
Then $c=2$, the starting points are $\{M\alpha\}$ and $1-\{\alpha\}$,
and the finish points are $1-\{M\alpha\}$ and $\{\alpha\}$. Now
write $\mathbb{D}\subset\mathbb{B}$ for the set of points $\{||m\alpha||\mid1\leq m\leq N\}$.
Since $\alpha$ is irrational, the points of $\mathbb{B}$ are all
distinct. Arrange the points in $\mathbb{D}$ in usual increasing
order. Since the ordering $\prec$ on $\mathbb{B}$ is the usual order on 
the circle $[0,1]$, and since $\mathbb{B} {\backslash} \mathbb{D}\subset (\frac{1}{2},1)$, 
it follows that if $u,v$ are consecutive points of $\mathbb{D}$, then it
they are also consecutive points of $\mathbb{B}$. Consequently, it 
follows from Theorem \ref{thm1} that the number of distinct gap 
values in $\mathbb{D}$ is at most
$3c=6$. 

\end{proof}
\begin{rem}
When $\alpha$ is a positive cube root of $15$, we get four distinct
gap sizes: 0.$000612999$, $0.006205886$, $0.006818885$, $0.007125385$.
Henk Don \cite{henk} has shown that the bound is precisely $4$. 
\end{rem}

\begin{proof}[Proof of Corollary \ref{thm:2}] 
Let $\coprod_{i=1}^{d}I_{i}$ be  a partition of the interval
$[0,M\alpha]$ into smallest possible number of connected parts such that $f| I_{i}=\hat{f}_{i}| I_{i}$
for some linear functions $\hat{f}_{i}:x\in\mathbb{R}\mapsto\frac{p_{i}}{q}x+k_{i}\in\mathbb{R}$,
$1\leq i\leq d$, where $0\neq p_{i}\in\mathbb{Z}$, $q\in\mathbb{Z}_{>0}$
and $k_{i}\in\mathbb{R}$. Let $n_i\leq N_i$ be uniquely defined integers such that 
$m\alpha \in I_i$ iff $n_i <m\leq N_i$, $1\leq i \leq d$. The result then follows from Theorem \ref{thm1}
by putting $N=M$ and $\{-\}_\lambda=\{-\}'_\infty$.

\end{proof}

\section{Acknoledgement}

The authors would like to thank Deepa Sahchari for
helpful discussions  and Tian An Wong for pointing out the reference \cite{henk}. They would especially like to thank the anonymous referee for pointing out a serious error in an earlier draft of this article because of which the statements of Theorem \ref{thm1} and Corollary \ref{thm:2} had to be modified.


\begin{bibdiv}
\begin{biblist}


\bib{BAASJ17}{article}{
    AUTHOR = {Balog, Antal}
    author={Granville, Andrew}
    author={Solymosi, Jozsef}
     TITLE = {Gaps between fractional parts, and additive combinatorics},
   JOURNAL = {Q. J. Math.},
  FJOURNAL = {The Quarterly Journal of Mathematics},
    VOLUME = {68},
      YEAR = {2017},
    NUMBER = {1},
     PAGES = {1--11},
      ISSN = {0033-5606},
   MRCLASS = {11B30 (11K06)},
  MRNUMBER = {3658281},
MRREVIEWER = {Robert F. Tichy},
       DOI = {10.1093/qmath/hav012},
       URL = {https://doi.org/10.1093/qmath/hav012},
}


\bib{grahamchung1976}{article}{   title={On the set of distances determined by the union of arithmetic progressions},   author={Chung, FRK}
author={Graham, RL}
journal={Ars Combinatoria},   volume={1},   pages={57--76},   year={1976},   publisher={Department of Combinatorics and Optimization, University of Waterloo.} }

\bib{henk}{article} {
    AUTHOR = {Don, Henk},
     TITLE = {On the distribution of the distances of multiples of an
              irrational number to the nearest integer},
   JOURNAL = {Acta Arith.},
  FJOURNAL = {Acta Arithmetica},
    VOLUME = {139},
      YEAR = {2009},
    NUMBER = {3},
     PAGES = {253--264},
      ISSN = {0065-1036},
   MRCLASS = {11J71 (11K06)},
  MRNUMBER = {2545929},
MRREVIEWER = {Luis Manuel Navas Vicente},
       DOI = {10.4064/aa139-3-4},
       URL = {https://doi.org/10.4064/aa139-3-4},
}


\bib{liang1979}{article}{title={A short proof of the 3d distance theorem},   author={Liang, Frank M},   journal={Discrete mathematics},   volume={28},   number={3},   pages={325--326},   year={1979},   publisher={Elsevier Science Publishers BV} }

\bib{Jens-Andreas2017}{article}{
    AUTHOR = {Marklof, Jens}
    author={Str\"{o}mbergsson, Andreas},
     TITLE = {The three gap theorem and the space of lattices},
   JOURNAL = {Amer. Math. Monthly},
  FJOURNAL = {American Mathematical Monthly},
    VOLUME = {124},
      YEAR = {2017},
    NUMBER = {8},
     PAGES = {741--745},
      ISSN = {0002-9890},
   MRCLASS = {11K06 (11H06 52C05)},
  MRNUMBER = {3706822},
       DOI = {10.4169/amer.math.monthly.124.8.741},
       URL = {https://doi.org/10.4169/amer.math.monthly.124.8.741},
}




\bib{sos1957theory}{article}{   title={On the theory of diophantine approximations I (on a problem of A. Ostrowski)},   author={S{\'o}s, Vera T},   journal={Acta Mathematica Hungarica},   volume={8},   number={3-4},   pages={461--472},   year={1957},   publisher={Akad{\'e}miai Kiad{\'o}}}



\bib{sos1958distribution}{article}{
  title={On the distribution mod 1 of the sequence {$n\alpha$} }
  author={{S}{\'o}s, Vera T},
  journal={Ann. Univ. {\"E}cient. Budapest E{\"o}tv{\"o}s {\"E}ect. Math,},
  volume={1},
  pages={127--134},
  year={1958}
}

\bib{sur1958nyi}{article}{
  title={\"Uber die Anordnung der Vielfachen einer reellen Zahl mod 1},
  author={Sur\'anyi, J\'anos},
  journal={Ann. Univ. Sci. Budupest E\"otv\"s Sect. Math},
  volume={1},
  PAGES = {107--111}
  year={1958}
}


\bib{swierczkowski1958successive}{article}{   title={On successive settings of an arc on the circumference of a circle},   author={{\'S}wierczkowski, S},   journal={Fundamenta Mathematicae},   volume={46},   pages={187--189},   year={1958},   publisher={Instytut Matematyczny Polskiej Akademii Nauk} }















\end{biblist}
\end{bibdiv}
\end{document}